\documentclass[11pt]{amsart}
\usepackage{amssymb, amstext, amscd, amsmath}
\usepackage{mathtools, xypic, color}
\usepackage{enumitem}
\usepackage{xcolor}
\usepackage{hyperref}
\theoremstyle{plain}
\newtheorem{thm}{Theorem}[section]

\newtheorem{cor}[thm]{Corollary}

\newtheorem{theorem}[thm]{Theorem}
\newtheorem{proposition}[thm]{Proposition}
\newtheorem{lemma}[thm]{Lemma}
\newtheorem{corollary}[thm]{Corollary}

\theoremstyle{definition}

\newtheorem{defn}[thm]{Definition}

\mathtoolsset{centercolon}
\newcommand{\bB}{{\mathbb{B}}}
\newcommand{\bC}{{\mathbb{C}}}
\newcommand{\bD}{{\mathbb{D}}}

\newcommand{\bN}{{\mathbb{N}}}

\newcommand{\bS}{{\mathbb{S}}}
\newcommand{\bT}{{\mathbb{T}}}

  \newcommand{\A}{{\mathcal{A}}}

\renewcommand{\H}{{\mathcal{H}}}
  
  \newcommand{\J}{{\mathcal{J}}}
  \newcommand{\K}{{\mathcal{K}}}  

  \newcommand{\M}{{\mathcal{M}}}
  \newcommand{\N}{{\mathcal{N}}}

  \newcommand{\U}{{\mathcal{U}}}

\newcommand{\fM}{{\mathfrak{M}}}

\newcommand{\fW}{{\mathfrak{W}}}

\renewcommand{\phi}{\varphi}

\newcommand{\upchi}{{\raise.35ex\hbox{$\chi$}}}

\newcommand{\ol}{\overline}

\newcommand{\qand}{\quad\text{and}\quad}

\newcommand{\qfor}{\quad\text{for}\quad}
\newcommand{\qforal}{\quad\text{for all}\quad}

\begin{document}
\title[Absolute continuity for commuting row contractions]{Absolute continuity for commuting row contractions}

\author[R. Clou\^atre]{Rapha\"el Clou\^atre}
\address{Department of Mathematics, University of Manitoba, 186 Dysart Road,
Winnipeg, Manitoba, Canada R3T 2N2}
\email{raphael.clouatre@umanitoba.ca\vspace{-2ex}}
\thanks{The first author was partially supported by an FQRNT postdoctoral fellowship and a start-up grant from the University of Manitoba.}

\author[K.R. Davidson]{Kenneth R. Davidson}
\address{Department of Pure Mathematics, University of Waterloo, 200 University Avenue West, 
Waterloo, ON, Canada N2L 3G1}
\email{krdavids@uwaterloo.ca}
\thanks{The second author is partially supported by an NSERC grant.}

\begin{abstract}
Absolutely continuous commuting row contractions admit a weak-$*$ continuous functional calculus. Building on recent work describing the first and second dual spaces of the closure of the polynomial multipliers on the Drury-Arveson space, we give a complete characterization of these commuting row contractions in measure theoretic terms.  We also establish that completely non-unitary row contractions are necessarily absolutely continuous, in direct parallel with the case of a single contraction. Finally, we consider refinements of this question for row contractions that are annihilated by a given ideal.
\end{abstract}

\maketitle

\section{Introduction} \label{S:intro}
A single contraction $T$ acting on a Hilbert space $\H$ can be successfully analyzed by using the theory developed by Sz.-Nagy and Foias in their seminal work  \cite{NagyFoias}. Briefly, one splits the contraction into a direct sum $T=T_{\text{cnu}}\oplus U$ where $U$ is a unitary operator and $T_{\text{cnu}}$ is \emph{completely non-unitary} in the sense that it has no closed invariant subspace on which it restricts to be a unitary operator. The study of $T$ then reduces to the separate examination of the two pieces $T_{\text{cnu}}$ and $U$. The unitary part $U$ is rather well understood by virtue of the classical spectral theorem, and thus one is left with the task of understanding the completely non-unitary part. Fortunately, the geometric structure of the minimal unitary dilation of completely non-unitary contractions is especially transparent (Proposition II.1.4 in \cite{NagyFoias}). It can be deduced thereof that  the polynomial functional calculus associated to a completely non-unitary contraction $T$ extends to a unital algebra homomorphism
\[
\Phi_T:H^\infty(\bD)\to B(\H)
\]
which is completely contractive and weak-$*$ continuous. This is the celebrated Sz.-Nagy--Foias $H^\infty(\bD)$ functional calculus. In fact, this functional calculus exists for every contraction whose minimal unitary dilation has a spectral measure which is absolutely continuous with respect to Lebesgue measure on the circle. This last observation rests on the complete characterization of the so-called Henkin measures on the circle (\cite{ColeRange}, \cite{Henkin}, \cite{Rudin}). The existence of the map $\Phi_T$ has had a tremendous impact on single operator theory, making significant appearances in great advances in the invariant subspace problem \cite{BCP} and in the classification of $C_0$ contractions \cite{bercovici} for instance; and it remains of fundamental importance in current work. In particular, the functional calculus creates a bridge between single operator theory and function theory in the classical Hardy space $H^2(\bD)$.

A topic of modern interest is the simultaneous study of several operators, or multivariate operator theory. One particular aspect of it is concerned with commuting operators $T_1,\ldots,T_d$ acting on the same Hilbert space $\H$ with the property that the row operator $T=(T_1,\ldots,T_d)$, which maps the space $\H^{(d)}=\H \oplus \H\oplus \ldots \oplus \H$ into $\H$ in the natural way, is contractive. We then say that $T=(T_1,\ldots,T_d)$ is a \emph{commuting row contraction}. These have been the target of intense research in recent years, spurred on by the work of Arveson \cite{Arv98} who showed that commuting row contractions are deeply linked with a space of holomorphic functions on the open unit ball, now called the \emph{Drury-Arveson space}. In fact, there always exists a contractive functional calculus
\[
\Phi_T:\A_d\to B(\H)
\]
where $\A_d$ is the closure of the polynomials in the multiplier norm, a norm which dominates the usual supremum norm over the ball.
The resulting analogy with the single operator correspondence between contractions and function theory on the disc has proved quite fruitful, and this paper is a contribution to this program. 

More precisely, we are interested in the natural functional calculus $\Phi_T$ associated to a commuting row contraction $T$. We wish to determine when it extends weak-$*$ continuously to the appropriate version of the algebra $H^\infty(\bD)$ in this context, namely the full multiplier algebra $\M_d$. This problem was  investigated by Eschmeier \cite{eschmeier97} for the more restrictive class of commuting row contractions which satisfy von Neumann's inequality for the supremum norm over the ball. In that paper, the multiplier algebras $\A_d$ and $\M_d$ were replaced by the more familiar objects $A(\bB_d)$ and $H^\infty(\bB_d)$.  A measure theoretic analysis can then be performed, mirroring the classical situation of a single contraction. In the general situation that we are interested in however, the issues are more delicate and are not merely of a measure theoretic nature. Much of our analysis is based on our recent work \cite{CD} which partly elucidates these subtleties.

The paper is organized as follows. Section \ref{S:prelim} introduces the necessary background and preliminary results. Section \ref{S:abscont} is concerned with the topic of absolute continuity for commuting row contractions. We give a complete characterization of that property in terms of the spectral measure of a spherical unitary (Lemma \ref{L:spmeasure} and Theorem \ref{T:abscontdilation}). The statement is an exact analogue of the classical result which applies to single contractions, but the proof is drastically different. We also give a decomposition of a commuting row contraction into an absolutely continuous part and a ``singular" part (Theorem \ref{T:AS}). In Section 4, we examine the property of being completely non-unitary and show that it is sufficient for absolute continuity (Theorem \ref{T:abscontCNU}). Along the way, we obtain multivariate analogues of single operator results which do not seem to have been noticed before (Theorem \ref{T:Hudilation}). Finally, in Section \ref{S:constrained} we consider commuting row contractions for which the functional calculus has a non-trivial kernel. We obtain a refined version of Theorem \ref{T:AS} (Theorem \ref{T:constraineddilation}) which takes that extra piece of information into account. We also show  that the size of the common zero set of the functions which annihilate $T$ can be used to detect purity (Theorem \ref{T:constrainedpure}), an important dilation theoretic rigidity property that is stronger than absolute continuity.

\section{Preliminaries and background}\label{S:prelim}
\subsection{Dilation of commuting row contractions and the Drury-Arveson space}\label{SS:dilation}
Let $\H$ be a Hilbert space and let $T_1,\ldots,T_d \in B(\H)$ be commuting bounded linear operators. The operator $T=(T_1,\ldots,T_d)$ acts on $\H^{(d)}=\H\oplus \cdots\oplus \H$ in the following natural way: if $\xi=(\xi_1,\ldots,\xi_d)\in \H^{(d)}$, then
\[
T\xi=\sum_{k=1}^d T_k \xi_k.
\]
Then, $T$ is called a \emph{commuting row contraction} if it is contractive. Equivalently, we require that
\[
\sum_{k=1}^d T_k T_k^*\leq I .
\]
Just as the study of single contractions is interwoven with complex function theory on the open unit disc $\bD\subset \bC$ through the existence of an isometric co-extension (\cite[Theorem~I.4.1]{NagyFoias}), the study of commuting row contractions is related to complex function theory on the open unit ball $\bB_d\subset \bC^d$ by means of a co-extension theorem. Before stating it, we introduce the appropriate space of functions.

The \emph{Drury-Arveson space}, denoted by $H^2_d$, is the reproducing kernel Hilbert space on $\bB_d$ with kernel function given by 
\[
k(z,w)=\frac{1}{1-\langle z,w\rangle_{\bC^d}} \qfor z,w\in \bB_d.
\]
Elements of $H^2_d$ are holomorphic functions on $\bB_d$. This space arises in several different contexts, which greatly increases the tools available for its analysis. For example, $H^2_d$ is the symmetric Fock space in $d$-variables \cite{Arv98}, and thus a quotient of the full Fock space as well. It is also a complete Nevanlinna-Pick interpolation space \cite{quiggen93, mccullough92, DP98b} and a Besov-Sobolev space \cite{chen03}.  See the survey \cite{ShalitDAsurvey} for more information.

We say that a function $\phi:\bB_d\to \bC$ is a \emph{multiplier} is $\phi f\in H^2_d$ for every $f\in H^2_d$. Then the linear operator $M_{\phi}:H^2_d\to H^2_d$ given by 
\[
M_\phi f = \phi f \qfor f \in H^2_d
\]
is automatically continuous.  Denote by $\M_d$ the algebra of all multipliers on $H^2_d$. Whenever convenient, we identify a multiplier $\phi$ with the induced multiplication operator $M_{\phi}\in B(H^2_d)$. In particular, the \emph{multiplier norm} is defined as 
\[
\|\phi\| = \|\phi\|_{\M_d}=\|M_{\phi}\|_{B(H^2_d)} \qfor  \phi\in \M_d.
\]
This allows us to view $\M_d$ as an operator algebra on $H^2_d$. It should be noted that if we put
\[
\|\phi\|_{\infty}=\sup_{z\in \bB_d}|\phi(z)|
\]
then
\[
\|\phi\|_{\infty}\leq \|\phi\|_{\M_d}
\]
for each $\phi\in \M_d$.
However, the two norms are not comparable in general \cite{Arv98, DP98b}.
All polynomials are multipliers, and we denote by $\A_d$ the closure of the polynomials in the multiplier norm, so that 
\[
\bC[z_1,\ldots,z_d]\subset \A_d\subset \M_d.
\]
A straightforward calculation shows that 
\[
M_z=(M_{z_1},\ldots,M_{z_d})
\]
is a commuting row contraction, which is called the \emph{$d$-shift}. 

The dilation theorem for commuting row contractions as we state it is due to Arveson \cite{Arv98}. However this result has a long history. Drury \cite{Drury78} first established a von Neumann inequality using the $d$-shift. Later M\"uller and Vasilescu \cite{MV93} proved the full form of the dilation theorem. However in both cases, the $d$-shift was merely given as a family of weighted shifts, without context placing these shifts in the framework of an interesting Hilbert space of functions. The dilation theorem shows that from the point of view of dilation theory for commuting row contractions, the $d$-shift plays a role analogous to that of the usual isometric unilateral shift for single contractions.

Recall that a commuting $d$-tuple of operators $U=(U_1,\ldots,U_d)$ is called a \emph{spherical unitary} if each $U_k$ is normal and
\[
\sum_{k=1}^d U_k U_k^*=I.
\]

\begin{theorem}\label{T:arvdilation}
Let $T=(T_1,\ldots,T_d)$ be a commuting row contraction. Then, there exists a spherical unitary $U=(U_1,\ldots,U_d)$ acting on some Hilbert space $\U$ along with a cardinal number $\kappa$ with the property that $T^*$ is unitarily equivalent to $(M_{z}^{(\kappa)}\oplus U)^*|_{\H'}$ for some subspace $\H'\subset (H^2_d)^{(\kappa)}\oplus \U$ which is co-invariant for $M_{z}^{(\kappa)}\oplus U$.
\end{theorem}

We usually identify $\H$ with $\H'$ so that $\H\subset (H^2_d)^{(\kappa)}\oplus \U$. Furthermore, the co-extension can be chosen to act on the smallest reducing subspace for $M_{z}^{(\kappa)}\oplus U$ containing $\H$. In this case, the co-extension is essentially unique and is referred to as the \emph{minimal co-extension}.
Note also that the cardinal $\kappa$ is the dimension of the so-called defect space associated to $T$, but we will not require this more precise information. 

Drury's von Neumann inequality for commuting row contractions follows immediately from Theorem \ref{T:arvdilation} upon replacing the algebra $H^\infty(\bD)$ by $\M_d$. Indeed, if $T=(T_1,\ldots,T_d)$ is a commuting row contraction, then
\[
\|p(T_1,\ldots,T_d)\|\leq \|p\|_{\M_d} \qfor  p\in \bC[z_1,\ldots,z_d].
\]
The dilation theorem actually shows that the matrix version of this inequality is also valid, just as in the one variable case. In particular, there exists a unital completely contractive algebra homomorphism
\[
\Phi_T:\A_d\to B(\H)
\]
such that $\Phi_T(z_k)=T_k$ for $1\leq k \leq d$.

\subsection{The dual space of $\A_d$}\label{SS:Henkin}

Recall now that $\M_d\subset B(H^2_d)$. An important property of $\M_d$ is that it is closed in the weak-$*$ topology of $B(H^2_d)$. It is easily verified that a sequence $\{f_n\}_n\subset \M_d$ converges to $0$ in the weak-$*$ topology if and only if it is bounded in the multiplier norm and converges pointwise to $0$ on the open unit ball $\bB_d$.

The basic motivation underlying this project is to determine  which commuting row contractions $T=(T_1,\ldots,T_d)$ have the property that the  associated functional calculus $\Phi_T$ extends  to a weak-$*$ continuous contractive map on the whole multiplier algebra $\M_d$.  
By analogy with the single variable case, and also with the non-commutative row contractive case of  \cite{DLP}, we make the following definition. We will provide in Corollary \ref{C:abscontchar} below the measure theoretic counterpart to this analogy. 

\begin{defn}
A commuting row contraction $T=(T_1,\ldots,T_d)$ is \emph{absolutely continuous} if the  associated functional calculus $\Phi_T$ extends to a weak-$*$ continuous map on the whole multiplier algebra $\M_d$. 
\end{defn}

Such an extension, if it exists, is unique and is automatically a completely contractive algebra homomorphism. 

We will be interested in the dual space $\A_d^*$. This was the object of interest in \cite{CD}, and the following is one of the main results contained therein.
Recall that given an operator algebra $\A\subset B(\H)$, the second dual $\A^{**}$ also has a natural operator algebra structure (see \cite[Section 2.5]{blecherlemerdy} for details).

\begin{theorem}\label{T:dual}
There exist a commutative von Neumann algebra $\fW$ and a completely isometric algebra isomorphism
\[
\Theta: \A_d^{**}\to \M_d \oplus \fW.
\]
Moreover, $\Theta$ is a weak-$*$ continuous homeomorphism and induces an isometric identification
\[
\Theta_*: \A_d^*\to\M_{d*}\oplus_1 \fW_*.
\]
\end{theorem}

The existence of $\Theta$ was established in Theorem 4.2 in \cite{CD}. The fact that it is a weak-$*$ homeomorphism was shown in the proof of Corollary 4.3 in \cite{CD}, which also shows the existence of $\Theta_*$.

Accordingly, we identify $\A_d^*$ with $\M_{d*}\oplus_1 \fW_*$ henceforth. Elements of $\M_{d*}\subset \A_d^*$ extend to be weak-$*$ continuous on $\M_d$. For that reason, they are called \emph{$\A_d$--Henkin functionals} by analogy with the case of the ball algebra (see \cite[Chapter 9]{Rudin}). 
At the other extreme are the elements of $\fW_*$ which are called \emph{$\A_d$--totally singular functionals}. Note in particular that the only functional that is $\A_d$--Henkin and $\A_d$--totally singular is the zero functional.

The next result is a combination of Theorems 3.3 and 4.4 from \cite{CD}. We let $\bS_d$ denote the unit sphere of $\bC^d$.

\begin{theorem}\label{T:measures}
Let $\Phi\in \A_d^*$. The following statements hold.
\begin{enumerate}
\item[\rm{(i)}] $\Phi\in \M_{d*}$ if and only if 
\[
\lim_{n\to \infty}\Phi(f_n)=0
\]
whenever $\{f_n\}_n\subset \A_d$ converges to $0$ in the weak-$*$ topology of $\M_d$.
\item[\rm{(ii)}] If $\Phi\in \fW_{*}$, then there exists a regular Borel measure $\mu$ on $\bS_d$ such that $\|\mu\|=\|\Phi\|$ and 
\[
\Phi(f)=\int_{\bS_d}f\,d\mu \qfor  f\in \A_d.
\]
\end{enumerate}
\end{theorem}

A standard argument based on part (i) of Theorem \ref{T:measures} shows that a commuting row contraction $T=(T_1,\ldots,T_d)$ is absolutely continuous if and only if for every sequence $\{f_n\}_n\subset \A_d$ converging to $0$ in the weak-$*$ topology of $\M_d$, we have that $\{f_n(T)\}_n$ converges to $0$ in the weak-$*$ topology (see the proof of Lemma 1.1 of \cite{eschmeier97}). This observation will be used throughout the paper.

The measures giving rise to functionals in $\fW_*$ via integration are said to be $\A_d$--totally singular, while those giving rise to functionals in $\M_{d*}$ are called $\A_d$--Henkin. These two sets of measures form complementary bands of measures on $\bS_d$, as the next result shows (see Theorem 5.4 in \cite{CD}).

\begin{theorem}\label{T:bands}
Let $\mu,\nu$ be regular Borel measures on $\bS_d$ such that $\mu$ is absolutely continuous with respect to $\nu$. If $\nu$ is $\A_d$-totally singular then so is $\mu$. If $\nu$ is $\A_d$-Henkin then so is $\mu$.
Moreover every measure on $\bS_d$ decomposes uniquely as the sum of an $\A_d$-Henkin measure and an $\A_d$-totally singular measure.
\end{theorem} 

Finally, it is important to note that since the multiplier norm and the supremum norm do not coincide, there are $\A_d$--Henkin functionals that are not given by integration against a measure. This phenomenon makes part of our analysis more complicated than that in \cite{eschmeier97}.

\section{A characterization of absolute continuity} \label{S:abscont}

In this section we establish our main results characterizing absolutely continuous commuting  row contractions. To establish these characterizations, we need to delve deeper into each summand of $\A_d^{*}$ and $\A_d^{**}$. First, we analyze the summand $\M_{d*}$ a bit further. The following lemma is not surprising.

\begin{lemma}\label{L:spmeasure}
Let $U=(U_1,\ldots,U_d)$ be a spherical unitary acting on some separable Hilbert space. Then, $U$ is absolutely continuous if and only if its spectral measure is $\A_d$--Henkin.
\end{lemma}

\begin{proof}
Let $\fM=W^*(U_1,\ldots,U_d)$.
There is a weak-$*$ homeomorphic $*$-isomorphism
\[
\Theta: \fM\to L^\infty(X,\mu)
\] 
such that $\Theta(U_k)=z_k$,  where the probability measure $\mu$ on $X\subset \bS_d$ is the scalar spectral measure of $U$. By the remark following Theorem \ref{T:measures}, $U$ is absolutely continuous if and only if, for every sequence $\{f_n\}_n\subset \A_d$ converging to $0$ in the weak-$*$ topology of $\M_d$, we have that $\{f_n(U)\}_n$ converges to $0$ in the weak-$*$ topology of $\fM$. Equivalently, this last condition says that
\[
\lim_{n\to \infty} \int_{X} f_n h\,d\mu=0
\]
for every $h\in L^1(\mu)$. Clearly, this is equivalent to $\mu$ being $\A_d$--Henkin, by virtue of Theorem \ref{T:bands}.
\end{proof}

We now obtain a characterization of the absolute continuity of a commuting row contraction in terms of the spherical unitary part of its minimal co-extension.

\begin{theorem}\label{T:abscontdilation}
Let $T=(T_1,\ldots,T_d)$ be commuting row contraction and let $U=(U_1,\ldots, U_d)$ be the spherical unitary part of its minimal co-extension. Then, $T$ is absolutely continuous if and only if $U$ is absolutely continuous. Furthermore, if $T$ acts on a separable Hilbert space, then these conclusions are equivalent to the spectral measure of $U$ being $\A_d$--Henkin.
\end{theorem}

\begin{proof}
By Theorem \ref{T:arvdilation}, up to unitary equivalence we have that
\[
T_k^*=(M_{z_k}^{(\kappa)}\oplus U_k)^*|\H \qfor  1 \leq k \leq d
\]
for some cardinal $\kappa$. By minimality, we have that $M_z^{(\kappa)}\oplus U$ acts on a space $\K=\K_1\oplus \K_2$ which coincides with the smallest reducing subspace containing $\H$. In particular, $\K$ is separable if $\H$ is separable.

Assume first that $U$ is absolutely continuous. Since $M_z^{(\kappa)}$ is trivially absolutely continuous and restriction to a co-invariant subspace preserves absolute continuity, it follows that $T$ is absolutely continuous as well.

Assume conversely that $T$ is absolutely continuous. We must show that $U$ is absolutely continuous. Since $\H$ is co-invariant under $M_z^{(\kappa)}\oplus U$, we have that
\[
\K_1\oplus \bigvee_{\alpha\in \bN^d} U^{\alpha} P_{\K_2}\H
\]
is a reducing subspace of $\K$ containing $\H$, and thus it coincides with $\K$. Let $\{f_n\}_n\subset \A_d$ be a sequence converging to $0$ in the weak-$*$ topology of $\M_d$. In particular, $\{f_n\}_n$ is bounded and so is $\{f_n(U)\}_n$, so that to show that $\{f_n(U)\}_n$ converges to $0$ in the weak-$*$ topology, it suffices to show that it converges to $0$ in the weak-operator topology. In fact, it suffices to establish that
\[
\lim_{n\to \infty} \langle f_n(U)U^{\alpha}P_{\K_2}\xi,U^\beta P_{\K_2}\eta \rangle=0
\]
for every pair of multi-indices $\alpha, \beta \in \bN^d$ and every $\xi,\eta\in \H$. Note that
\begin{align*}
\langle f_n(U)U^{\alpha}P_{\K_2}\xi, &U^\beta P_{\K_2}\eta \rangle \\
&=\langle f_n(U)U^{\beta*}P_{\K_2}\xi,U^{\alpha*} P_{\K_2}\eta \rangle\\\
&= \langle f_n(M_z^{(\kappa)}\oplus U)(M_z^{(\kappa)}\oplus U)^{\beta*}\xi,(M_z^{(\kappa)}\oplus U)^{\alpha*} \eta \rangle\\
&\qquad-\langle f_n(M_z^{(\kappa)})M_z^{(\kappa)\beta*}P_{\K_1}\xi,M_z^{(\kappa)\alpha*} P_{\K_1}\eta \rangle\\
&=\langle f_n(T)(M_z^{(\kappa)}\oplus U)^{\beta*}\xi,(M_z^{(\kappa)}\oplus U)^{\alpha*} \eta \rangle\\
&\qquad-\langle f_n(M_z^{(\kappa)})M_z^{(\kappa)\beta*}P_{\K_1}\xi,M_z^{(\kappa)\alpha*} P_{\K_1}\eta \rangle
\end{align*}
where the last equality follows since $(M_z^{(\kappa)}\oplus U)^{\beta*}\xi$ and $(M_z^{(\kappa)}\oplus U)^{\alpha*} \eta$ belong to $\H$. Since $M_z^{(\kappa)}$ and $T$ are absolutely continuous, we conclude that
\[
\lim_{n\to \infty} \langle f_n(U)U^{\alpha}P_{\K_2}\xi,U^\beta P_{\K_2}\eta \rangle=0 ;
\]
and thus $U$ is absolutely continuous.

The final statement follows from Lemma \ref{L:spmeasure}.
\end{proof}

Next, we focus on the summand $\fW\subset \A_d^{**}$. 
Given $f\in \A_d$, we denote by $\widehat{f}$ its canonical image in $\A_d^{**}$. Let $E_1=1\oplus 0, E_2=0\oplus 1$ which both lie in  $\A_d^{**}$. Hence, multiplication on the left or on the right by $E_k$ is the projection onto the $k$-th component of $\A_d^{**}$. These projections are weak-$*$ continuous by \cite{CER}. They have another important property, which is contained in the following lemma. 

Recall first that any bounded linear map $\rho:\A_d\to B(\H)$ has a unique weak-$*$ continuous extension $\widehat{\rho}:\A_d^{**}\to B(\H)$ defined as $\widehat{\rho}=\iota^*\circ \rho^{**}$ where 
\[
\iota:B(\H)_*\to B(\H)^*
\]
is the canonical embedding. If $\rho$ is multiplicative, then so is $\widehat{\rho}$.

\begin{lemma}\label{L:Phias}
Let $\Phi\in \A_d^*$ and let $\widehat{\Phi}:\A_d^{**}\to \bC$ be its unique weak-$*$ continuous extension. Define $\Phi_a,\Phi_s\in \A_d^*$ as follows
\[
\Phi_a(f)=\widehat{\Phi}(E_1 \widehat{f}), \quad \Phi_s(f)=\widehat{\Phi}(E_2 \widehat{f}), \quad f\in \A_d.
\]
Then, $\Phi_a$ extends to be weak-$*$ continuous on $\M_d$ and $\Phi_s\in \fW_*$.
\end{lemma}

\begin{proof}
This is contained in Corollary 4.3 (and its proof) from \cite{CD}.
\end{proof}

We now present a decomposition result for contractive representations of $\A_d$.

\begin{lemma}\label{L:decomprep}
Let
\[
\rho:\A_d\to B(\H)
\]
be a contractive homomorphism. Then, 
\[
\rho(f)=\widehat{\rho}(E_1 \widehat{f})\oplus \widehat{\rho}(E_2 \widehat{f}) \qfor  f\in \A_d.
\]
Moreover, if $\Phi$ is a weak-$*$ continuous functional, then the map
\[
f\mapsto \Phi(\widehat{\rho}(E_1\widehat{f})), \quad  f\in \A_d
\]
extends to be weak-$*$ continuous on $\M_d$ while
\[
f\mapsto \Phi(\widehat{\rho}(E_2\widehat{f}),  \quad  f\in \A_d
\]
belongs to $\fW_*$.
\end{lemma}

\begin{proof}
Let $Q_k= \widehat{\rho}(E_k)$ for $k=1,2$. Note that $Q_1Q_2=Q_2Q_1=0$. Also, each $Q_k$ is a contractive idempotent, and thus a self-adjoint projection. Moreover, $Q_k$ commutes with $\widehat{\rho}(\A_d^{**})$ and hence $Q_k \H$ is a reducing subspace for $\widehat{\rho}(\A_d^{**})$. Since $\widehat{\rho}$ is multiplicative we find 
\[
\rho(f)=Q_1\widehat{\rho}(E_1\widehat{f})Q_1\oplus Q_2\widehat{\rho}(E_2 \widehat{f})Q_2=\widehat{\rho}(E_1 \widehat{f})\oplus \widehat{\rho}(E_2 \widehat{f})
\]
for each $f\in \A_d$.  The last two statements follow from Lemma \ref{L:Phias}.
\end{proof}

Specializing the previous lemma to the case of the functional calculus associated to a commuting row contraction will yield the  main result of this section. One last bit of preparation is necessary. Comparing statements, we see that the following result is similar in spirit to Theorem 1.3 from \cite{eschmeier97}. 

\begin{theorem}\label{T:Wunitary}
If we put
\[
u_k=E_2 \widehat{z_k} \qfor  1\leq k \leq d
\]
then $\fW=W^*(u_1,\ldots,u_d)$ and $u=(u_1,\ldots,u_d)$ is a spherical unitary.
\end{theorem}
\begin{proof}
Note that $\fW=E_2 \A_d^{**}$. Since the algebra $\A_d^{**}$ is the weak-$*$ closed algebra generated by $\{\widehat{z_k}:1\leq k \leq d\}$, it follows that $\fW=W^*(u_1,\ldots,u_d)$. We now turn to showing that $u$ is a spherical unitary.
Since the $d$-shift is a row contraction and the canonical embedding $\A_d\to \A_d^{**}$ is completely isometric (Proposition 1.4.1 in \cite{blecherlemerdy}), it follows that $u$ is a row contraction. Therefore the joint spectrum $X$ of $u$ is contained in $\ol{\bB_d}$. 
We will show that the spectrum is disjoint from $\bB_d$.

Suppose that $\A_d^{**}$ is represented completely isometrically as a subalgebra of some $B(\H)$. While the Hilbert space $\H$ may not be separable, we can find a decomposition
\[
\H=\bigoplus_{\alpha}\H_{\alpha}
\]
into separable reducing subspaces for $\fW$ such that
\[
\fW=\bigoplus_{\alpha}\fW_\alpha
\]
and each $\fW_{\alpha}$ acts on a separable space. Let $P_{\alpha}\in \fW'$ be the orthogonal projection onto $\H_{\alpha}$. It suffices to show that the joint spectrum of each $P_{\alpha}u$ lies on the sphere. Now, we see that
\[
\fW_{\alpha}=W^*(P_{\alpha}u_1,\ldots,P_{\alpha}u_d)
\]
and thus there is a weak-$*$ homeomorphic $*$-isomorphism 
\[
\Theta_{\alpha}: \fW_{\alpha}\to L^\infty(X_{\alpha},\mu_{\alpha})
\]
where $X_{\alpha}\subset \ol{\bB_d}$ is the joint spectrum of $u_{\alpha}$ while the Borel probability measure $\mu_{\alpha}$ is its spectral measure. Moreover,  $\Theta_{\alpha}(P_{\alpha}u_k)=z_k$ for each $1\leq k \leq d$.

Let $\Omega\subset \bB_d$ be a Borel subset. Let
 \[
 \Phi_{\alpha}: \fW\to \bC
 \]
be defined as
\[
\Phi_{\alpha}(w)=\int_{\bB_d}\Theta_{\alpha}(P_{\alpha}w)\chi_{\Omega}\,d\mu_{\alpha}
\]
for each $w\in \fW$. Note that 
\[
\chi_\Omega \mu_{\alpha}\in L^\infty(X_{\alpha},\mu_{\alpha})_*
\]
and thus $\Phi_{\alpha}$ is seen to be weak-$*$ continuous, so that  $\Phi_{\alpha}\in \fW_*$.
On the other hand, if $f$ is a polynomial we have
\begin{align*}
\Phi_{\alpha}(E_2 \widehat{f})&=\int_{\bB_d}f(\Theta_{\alpha}(P_{\alpha}u_1),\ldots,\Theta_{\alpha}(P_{\alpha}u_d))\chi_{\Omega}\,d\mu_{\alpha}\\
&=\int_{\bB_d}f\chi_{\Omega}\,d\mu_{\alpha}.
\end{align*}
In fact, this equality is easily seen to hold for $f\in \A_d$. Since $\Omega\subset \bB_d$, a straightforward verification based on the dominated convergence theorem shows that 
\[
f\mapsto \Phi_{\alpha}(E_2 \widehat{f})
\]
is an $\A_d$--Henkin functional. Since this map was also shown to be in $\fW_*$, we conclude that it is the zero map on $\A_d$. Hence
\begin{align*}
\mu_{\alpha}(\Omega)=\Phi_{\alpha}(E_2 \widehat{1})=0.
\end{align*}
Since $\Omega\subset \bB_d$ was arbitrary, we conclude that $X_{\alpha}\subset \bS_d$, and because $\alpha$ was arbitrary this implies that the joint spectrum of $u$ is contained in $\bS_d$.
\end{proof}


An alternate proof of this  theorem can be given by constructing explicit bounded nets which converge to each $u_k^*$,
by approximating the functions $\ol{z_k}$ on $\A_d$-totally null subsets of $\bS_d$ and simultaneously converging
to 0 uniformly on compact subsets of $\bB_d$. 
The construction is technical and relies on the interpolation result \cite[Proposition~5.7]{CD}. On the other hand, it does provide intuition as to why $\psi(1 - \sum_k u_k u_k^*) = 0$ for all $\psi\in\fW_*$. 
In turn, this leads to the conclusion that $\sum_k u_k u_k^* = 1$, which shows that $u=(u_1,\ldots,u_d)$ is a spherical unitary.

We make the following definition which is complementary to absolute continuity.

\begin{defn}\label{D:TS}
A commuting row contraction $T = (T_1,\dots,T_d)$ on a Hilbert space $\H$ is called \emph{totally singular} if for every weak-$*$ continuous linear functional $\Phi$ on $B(\H)$, the functional
\[
 f\mapsto \Phi(f(T)),\quad  f\in \A_d
 \]
belongs to $\fW_*$. 
\end{defn}

Equivalently, $T$ is totally singular whenever  
\[
 f\mapsto \Phi(f(T)), \quad  f\in \A_d
 \]
belongs to $\fW_*$ for every weak-$*$ continuous functional $\Phi$ on $W^*(T_1,\ldots,T_d)$.

We will give an alternative characterization of totally singular commuting row contractions after the next theorem, which is the main result of this section.
Similar statements were established using different tools by Eschmeier  \cite[Theorem 1.5]{eschmeier97} under stronger assumptions on the commuting row contraction $T$. 

\begin{thm} \label{T:AS}
Let $T = (T_1,\dots,T_d)$ be a commuting row contraction.
Then, we can write $T=A\oplus S$ where
\begin{enumerate}
\item[\rm{(i)}] 
$A=(A_1,\dots,A_d)$ is absolutely continuous, 
\item[\rm{(ii)}] 
$S=(S_1,\dots,S_d)$  is a totally singular spherical unitary.
\end{enumerate}
\end{thm}

\begin{proof}
Let $\rho:\A_d\to B(\H)$ be the unital completely contractive homomorphism defined by $\rho(f) = f(T)$ for every $f\in \A_d$. By Lemma \ref{L:decomprep}, we can write
\[
T_k=A_k\oplus S_k \qfor 1\leq k \leq d
\]
where 
\[
f(A_k)=\widehat{\rho}(E_1 \widehat{f})  \qand f(S_k)=\widehat{\rho}(E_2 \widehat{f}) 
\]
for  $f\in \A_d$ and $1\leq k \leq d$.
Now, $\widehat{\rho}|_{\fW}$ is a unital contractive homomorphism on the von Neumann algebra $\fW$, hence a $*$-homomorphism. In particular, 
$S=(S_1,\ldots,S_d)$ is a spherical unitary by Theorem \ref{T:Wunitary}, and furthermore it is totally singular by Lemma \ref{L:decomprep}. Finally, the fact that $A$ is absolutely continuous is a consequence of Lemma \ref{L:decomprep} and the remark following Theorem \ref{T:measures}.

\end{proof}

We make a few remarks. First, note that Theorem \ref{T:abscontdilation} implies that the operator $A$ in the theorem above has a co-extension of the form $M_z^{(\kappa)}\oplus U$ where $U$ is an absolutely continuous spherical unitary. Furthermore,
in the case where the commuting row contraction $T$ above acts on a separable space, the spherical unitary $S$ has a spectral measure which must be $\A_d$--totally singular.  In fact, we establish a full counterpart to Lemma \ref{L:spmeasure} for totally singular commuting row contractions.

\begin{corollary}\label{C:spmeasuresing}
Let $T=(T_1,\ldots,T_d)$ be a commuting row contraction on some separable Hilbert space. Then, $T$ is totally singular if and only if $T$ is a spherical unitary with $\A_d$--totally singular spectral measure.
\end{corollary}
\begin{proof}
Assume that $T$ is totally singular. By Theorem \ref{T:AS}, we see that $T$ must be a spherical unitary. Moreover, upon taking the functional $\Phi$ in Definition \ref{D:TS} to be integration against the spectral measure of $T$, we conclude that this measure must be $\A_d$--totally singular.

Conversely, assume that $T$ is a spherical unitary with $\A_d$--totally singular spectral measure. 
Let $\fM=W^*(T_1,\ldots,T_d)$. There is a weak-$*$ homeomorphic $*$-isomorphism
\[
\Theta: \fM\to L^\infty(X,\mu)
\] 
such that $\Theta(T_k)=z_k$,  where the probability measure $\mu$ on $X\subset \bS_d$ is the spectral measure of $T$. To show that $T$ is totally singular, it suffices to verify that
\[
f\mapsto \Phi(f(T))=\Phi\circ \Theta^{-1}(f)
\]
belongs to $\fW_*$ whenever $\Phi\in \fM_*$. In that case,  we see that $\Phi\circ \Theta^{-1}$ belongs to $L^\infty(X,\mu)_*$ and thus
\[
\Phi\circ \Theta^{-1}(f)=\int_{X}fh\,d\mu \qfor  f\in \A_d
\]
for some $h\in L^1(\mu)$. Since $\mu$ is $\A_d$--totally singular, so is $h\mu$ by Theorem \ref{T:bands}, so that $\Phi\circ \Theta^{-1}\in \fW_*$ and the proof is complete.
\end{proof}

We close this section with an analogue of Theorem 1.3 from \cite{eschmeier97} showing that the functional calculus associated to a totally singular commuting row contraction extends to a $*$-homomorphism.

\begin{cor}\label{C:*hom}
Let $T=(T_1,\dots,T_d)$ be a totally singular commuting row contraction on some Hilbert space $\H$. Then, there is a $*$-homomorphism $\pi:C(\bS_d)\to B(\H)$ such that $\pi(f)=f(T)$ for each $f\in \A_d$.
\end{cor}

\begin{proof}
By Corollary \ref{C:spmeasuresing}, we see that $T$ is a spherical unitary. We can thus simply take $\pi$ to be the continuous functional calculus for $T$.
\end{proof}

\section{Completely non-unitary commuting row contractions}

Every contraction $T$ acting on a separable Hilbert space has a decomposition into a completely non-unitary part and a unitary part. The completely non-unitary part is well-known to be absolutely continuous as mentioned in the introduction (Theorem II.6.4 in \cite{NagyFoias}). The unitary part can be further decomposed using the Lebesgue decomposition of its spectral measure into an absolutely continuous unitary and a ``singular" one. Thus
\[ T = T_{cnu} \oplus U_a \oplus U_s .\]
The decomposition of Theorem~\ref{T:AS} in the case $d=1$ thus yields $A = T_{cnu} \oplus U_a$ and $S = U_s$.
In this short section, we discuss the multivariate counterpart to this argument.

A commuting row contraction $T=(T_1,\ldots,T_d)$ is said to be completely non-unitary if it has no invariant subspace on which it restricts to a spherical unitary. The next result is folklore, but we provide the details as a proof of it is not easy to track down in the literature (\cite{eschmeier95}, \cite[Exercise 20.2.6]{DAELW}). We are indebted to J\"org Eschmeier for sharing the following proof with us \cite{eschmeier2015}.

\begin{theorem}\label{T:CNU}
Let $T=(T_1,\ldots,T_d)$ be a commuting row contraction on some Hilbert space $\H$. Then, there exist unique reducing subspaces $\H_{\text{cnu}}$ and $\H_u$ for $T$ such that $\H = \H_{\text{cnu}} \oplus \H_u$, $T|_{\H_{\text{cnu}}}$ is completely non-unitary and $T|_{\H_u}$ is a spherical unitary.
\end{theorem}

\begin{proof}
Let $\H_u$ be the smallest closed subspace of $\H$ containing all invariant subspaces $\N\subset \H$ for $T$ with the property that $T|_{\N}$ is a spherical unitary. Note that all such subspaces are necessarily reducing for $T$, and thus clearly $\H_u$ is reducing as well.
Let $\H_{\text{cnu}}=\H\ominus \H_u$. It is obvious that $T|_{\H_{\text{cnu}}}$ is completely non-unitary. Put $S=T|_{\H_u}$. It remains to check that $S$ is a spherical unitary. Observe that each subspace $\N$ on which $T$ restricts to a spherical unitary lies in the space
\[
 \ker\big(I-\sum_{k}T^*_k T_k\big)\ \cap\ \ker\big(I-\sum_{k}T_k T_k^*\big).
\]
Thus $\H_u$ lies in this intersection, which shows that $S$ satisfies
\begin{equation}\label{E:sphisom}
 \sum_k S_k^*S_k = I = \sum_k S_k S_k^* .
\end{equation}

By a classical result of Athavale \cite{athavale} (this also follows from Theorem \ref{T:arvdilation}), every spherical isometry extends to a spherical unitary. Now, $S$ is a spherical isometry by the first equality in (\ref{E:sphisom}), and hence has such an extension. But the second equality of (\ref{E:sphisom}) states that $S^*$ is also a spherical isometry, which forces the extension to be trivial (that is, a direct sum). Hence $S$ is a spherical unitary, and a fortiori, the operators $S_k$ are commuting normal operators.

Finally, the uniqueness of the decomposition is a straightforward consequence of the definitions of $\H_u$ and $\H_{\text{cnu}}$.
\end{proof}

Recall that  
\[
\lim_{N\to \infty}\sum_{|\alpha|=N} \frac{N!}{\alpha_1! \alpha_2!\cdots \alpha_d!}\|M_{z^\alpha}^{(\kappa)*}\xi\|^2 = 0
\]
for every $\xi\in (H^2_d)^{(\kappa)}$ \cite{Arv98}. The following alternate description of $\H_u$ is inspired by the proof of Proposition I.1.7 in \cite{bercovici} and that of Theorem 1 in \cite{eschmeier95}. It appears to be new in our context.

\begin{theorem}\label{T:Hudilation}
Let $T=(T_1,\ldots,T_d)$ be a commuting row contraction on some Hilbert space $\H$ with minimal co-extension $M_z^{(\kappa)}\oplus U$ acting on $\K_1\oplus \K_2$. 
Assume that $\H_u\subset \H$ is a reducing subspace for $T$ such that $T|_{\H_u}$ is a spherical unitary and $T|_{\H\ominus \H_u}$ is completely non-unitary.
Then
\[
\H_u=\K_2\ominus \bigvee_{\alpha\in \bN^d}U^{\alpha*}(\K_2\ominus (\H\cap \K_2)).
\]
\end{theorem}

\begin{proof}
Denote by $\H_0$ the space on the right-hand side of the desired equality. Note that 
\[
\H_0\subset  \K_2\ominus (\K_2\ominus (\H\cap \K_2)) = \H \cap \K_2 .
\]
By construction, $\H$ is co-invariant for $M_z^{(\kappa)}\oplus U$ and $\K_2$ is reducing. Therefore $\H \cap \K_2$ is co-invariant; whence $\K_2\ominus (\H\cap \K_2)$ is invariant. Since 
\[
(M_z^{(\kappa)}\oplus U)|_{\K_2} = U,
\]
we see that 
\[
\bigvee_{\alpha\in \bN^d}U^{\alpha*}(\K_2\ominus (\H\cap \K_2))
\]
is the smallest reducing subspace for $M_z^{(\kappa)}\oplus U$ containing $\K_2\ominus (\H\cap \K_2)$. 
It follows that $\H_0$ is the largest reducing subspace  for $M_z^{(\kappa)}\oplus U$ contained in $\H \cap \K_2$. 
In particular, $P_{\H_0}$ commutes with $M_z^{(\kappa)}\oplus U$.
Therefore 
\[ T P_{\H_0} = P_\H (M_z^{(\kappa)} \oplus U) P_{\H_0} = P_\H  P_{\H_0} (M_z^{(\kappa)} \oplus U) = P_{\H_0} U = UP_{\H_0}\]
and $T|_{\H_0}$ is a spherical unitary. This shows that $\H_0 \subset \H_u$ by uniqueness of $\H_u$ (see Theorem \ref{T:CNU} and its proof).

Next, we claim that $\H_u$ is contained in $\K_2$. To this end, let $\xi\in \H_u$. Then,
\begin{align*}
\|\xi\|^2&=\lim_{N\to \infty}\sum_{|\alpha|=N} \frac{N!}{\alpha_1! \alpha_2!\cdots \alpha_d!} \|T^{\alpha*}\xi\|^2\\
&=\lim_{N\to \infty}\sum_{|\alpha|=N} \frac{N!}{\alpha_1! \alpha_2!\cdots \alpha_d!} \left(\|M_{z^\alpha}^{(\kappa)*} P_{\K_1}\xi\|^2 + \|U^{\alpha*}P_{\K_2}\xi\|^2\right)\\
&=\lim_{N\to \infty}\sum_{|\alpha|=N} \frac{N!}{\alpha_1! \alpha_2!\cdots \alpha_d!}\|U^{\alpha*}P_{\K_2}\xi\|^2=\|P_{\K_2}\xi\|^2,
\end{align*}
so that $\xi\in \K_2$. This shows that $\H_u\subset \H\cap \K_2$. Moreover, since $\H_u$ is reducing for $T$, it must be invariant for $U^*$ because
\[
T^*|_{\H_u}=(M_z^{(\kappa)}\oplus U)^*|_{\H_u}=U^*|_{\H_u}.
\]
Since the row contraction $(U_1^*,\ldots,U_d^*)$ restricts to a spherical unitary on $\H_u$, this subspace must be co-invariant for $U^*$, so in fact $\H_u$ is reducing for $M_z^{(\kappa)}\oplus U$. As $\H_0$ is the largest reducing subspace $M_z^{(\kappa)}\oplus U$ contained in $\H\cap\K_2$, we deduce that $\H_u\subset \H_0$. Therefore $\H_u = \H_0$ as claimed.
\end{proof}

A commuting row contraction $T=(T_1,\ldots,T_d)$ acting on $\H$ is said to be \emph{pure} if
\[
\lim_{N\to \infty}\sum_{|\alpha|=N} \frac{N!}{\alpha_1! \alpha_2!\cdots \alpha_d!}\|T^{\alpha*}\xi\|^2=0 \qforal \xi\in \H .
\]
We noted before the proof of Theorem \ref{T:Hudilation} that $M_z^{(\kappa)}$ is pure for any cardinal number $\kappa$. In particular, a general commuting row contraction $T$ is pure if and only if its minimal co-extension has no spherical unitary part. Since $M_z$ is trivially absolutely continuous, all pure row contractions are absolutely continuous: this is the easy part of Theorem \ref{T:abscontdilation}. To guarantee absolute continuity, purity can actually be weakened to the commuting row contraction merely being completely non-unitary. This was also observed in Corollary 1.7 of \cite{eschmeier97} under stronger conditions.

\begin{theorem}\label{T:abscontCNU}
Let $T$ be a completely non-unitary commuting row contraction. Then, $T$ is absolutely continuous.
\end{theorem}

\begin{proof}
By Theorem \ref{T:AS}, we can write $T=A\oplus S$ where $A$ is absolutely continuous and $S$ is a spherical unitary. By assumption, we see that this last summand must be absent.
\end{proof}

\section{Constrained  commuting row contractions}\label{S:constrained}

Let $T=(T_1,\ldots,T_d)$  be a commuting row contraction. As mentioned in Section \ref{S:prelim}, there is a unital completely contractive algebra homomorphism 
\[
\Phi_T:\A_d\to B(\H)
\]
with the property that $\Phi_T(z_k)=T_k$ for $1\leq k \leq d$.
Let $\J\subset \A_d$ be a closed ideal. Then, $T$ is said to be \emph{$\J$-constrained} if $f(T)=0$ for every $f\in \J$.
These objects were introduced and studied by Popescu \cite{Pop06}.
Our first goal in this section is to refine Theorem \ref{T:AS} for constrained commuting row contractions, using the additional information that the ideal $\J$ provides. The reader should note that this could be accomplished using Theorem 2.1 from \cite{Pop06}, but we provide a direct proof based on Theorem \ref{T:AS} and manage to extract more precise information.

Before proceeding with the result, let us establish some notation. Given a closed ideal $\J\subset \A_d$, we put
\[
V(\J) = \{z\in \ol{\bB_d}: f(z)=0 \qforal f\in \J\} ,
\]
\[
\J H^2_d=\{fh: f\in \J, h\in H^2_d\} \qand \N_{\J}=(\J H^2_d)^\perp.
\]
Note that the subspace $\N_\J\subset H^2_d$ is co-invariant for $M_z$. We also put
\[
Z_k=P_{\N_{\J}} M_{z_k}|_{\N_\J} \qfor  1\leq k \leq d.
\]
Then, $Z=(Z_1,\ldots,Z_d)$  is a $\J$-constrained commuting row contraction. For simplicity, we restrict our attention to separable Hilbert spaces.

\begin{thm} \label{T:constraineddilation}
Let $\J\subset \A_d$ be a closed ideal and  let $T = (T_1,\dots,T_d)$ be a $\J$-constrained commuting row contraction acting on some separable Hilbert space $\H$.
Then, we can write $T=A\oplus S$ where
\begin{enumerate}
\item[\rm{(i)}] 
 $S=(S_1,\dots,S_d)$  is a spherical unitary with an $\A_d$--totally singular spectral measure supported in $V(\J) \cap \bS_d$, and

\item[\rm{(ii)}] 
$A=(A_1,\dots,A_d)$ can be co-extended to $Z^{(\kappa)} \oplus W$ for some cardinal $\kappa$ and some spherical unitary $W=(W_1,\ldots,W_d)$ with an $\A_d$--Henkin spectral measure supported in $V(\J) \cap \bS_d$.
\end{enumerate}
\end{thm}

\begin{proof}
By Theorem \ref{T:AS} and the remarks following it, we can write $T=A\oplus S$ where $A=(A_1,\ldots,A_d)$ is an absolutely continuous commuting row contraction and $S=(S_1,\ldots,S_d)$ is a spherical unitary with $\A_d$--totally singular spectral measure. By virtue of Theorem \ref{T:abscontdilation}, we also have that the minimal co-extension of $A$ is of the form $M_z^{(\kappa)}\oplus U$ for some cardinal $\kappa$ and some spherical unitary $U$ with $\A_d$--Henkin spectral measure. Write $\H=\H_a\oplus \H_s$ where $A$ acts on $\H_a$ and $S$ acts on $\H_s$. For each $f \in \J$ we have
\begin{align*}
  0 &= f(T)^* = f(A)^*\oplus f(S)^*\\\
  &= ( f(M_z^{(\kappa)}) \oplus f(U) )^* |_{\H_a} \oplus f(S)^* 
\end{align*}
since $\H_a$ is co-invariant for $M_z^{(\kappa)}\oplus U$.
In particular, it follows that $f(S)=0$ for every $f\in \J$ and thus the support of the spectral measure of $S$ lies in $V(\J)\cap \bS_d$ by the spectral theorem.
In addition,  we see that 
\[
\H_a\subset \ker (M_f^{(\kappa)*} \oplus f(U)^* )
\]
for every $f\in \J$. In other words, we have
\[
\H_a\subset \bigcap_{f\in \J}\left(\ker  M_f^{(\kappa)*}\oplus \ker f(U)^*\right) =\N_\J^{(\kappa)}\oplus \U_\J
\]
where 
\[
\U_{\J}=\bigcap_{f\in \J}\ker f(U)^*
\]
which is reducing for $U$. In particular, we see that $W=U|_{\U_\J}$ is a spherical unitary with the property that $f(W)=0$ for every $f\in \J$. Since $W$ is a restriction of $U$ to a reducing subspace and $U$ is absolutely continuous, so is $W$ (in fact, by the minimality of the co-extension, $W=U$). Therefore, $W$ has an $\A_d$--Henkin spectral measure by Lemma \ref{L:spmeasure}, and that measure is supported on $V(\J)\cap \bS_d$. Finally, the fact that
\[
\H_a\subset \N_\J^{(\kappa)}\oplus \U_\J
\] 
implies that $A$ co-extends to $Z^{(\kappa)}\oplus W$ and the proof is complete.
\end{proof}

Given a closed ideal $\J\subset \A_d$, we denote by $\J_w\subset \M_d$ the closure of $\J$ in the weak-$*$ topology of $\M_d$. Clearly, $\J_w$ is an ideal of $\M_d$.
The following result yields information about the functional calculus $\Phi_T$ associated to a commuting row contraction $T$.

\begin{proposition} \label{P:constrainedfunctcalc}
Let $\J\subset \A_d$ be a closed ideal and let $T=(T_1,\ldots,T_d)$ be an absolutely continuous $\J$-constrained commuting row contraction acting on a separable Hilbert space $\H$. Then, there is a unital, weak-$*$ continuous, completely contractive algebra homomorphism
\[
\Psi_T:\M_d\to B(\H)
\]
which extends $\Phi_T$. Moreover, $\Psi_T(f)=0$ for every $f\in \J_w$ and $\Psi_T$ factors through $\M_d/ \J_w$. 
\end{proposition}

\begin{proof}
The existence of $\Psi$ follows from the very definition of absolute continuity of $T$. The second statement is clear. 
\end{proof}

Next, we obtain a refinement of Theorem \ref{T:abscontdilation} which could be proved along the same lines, but we give a short proof based on Theorem \ref{T:constraineddilation}.

\begin{corollary}\label{C:abscontchar}
Let $\J\subset \A_d$ be a closed ideal and let $T=(T_1,\ldots,T_d)$ be a $\J$-constrained commuting row contraction  acting on a separable Hilbert space $\H$. The following statements are equivalent.
\begin{enumerate}
\item[\rm{(i)}] $T$ is absolutely continuous

\item[\rm{(ii)}] $T$ co-extends to $Z^{(\kappa)} \oplus W$ for some cardinal $\kappa$ and some spherical unitary
$W$ with $\A_d$--Henkin spectral measure supported in $V(\J)\cap \bS_d$.
\end{enumerate}
\end{corollary}

\begin{proof}
The fact that (ii) implies (i) is a straightforward consequence of Lemma \ref{L:spmeasure}.
Assume thus that (i) holds. Invoking Theorem \ref{T:constraineddilation}, we see that $T=A\oplus S$  where $S$ is a spherical unitary with $\A_d$--totally singular spectral measure. On the other hand, $S$ must be absolutely continuous since $T$ is, and thus its spectral measure must be $\A_d$--Henkin by Lemma \ref{L:spmeasure}. Hence, that measure is zero and $S$ is absent.
\end{proof}

In general, absolute continuity does not imply purity: the usual bilateral shift on $L^2(\bS_1)$ provides an example of an absolutely continuous unitary operator. On the other hand, this statement is valid for certain constrained absolutely continuous contractions (see Lemma II.1.12 in \cite{bercovici}). We prove a multivariate analogue of that fact. Roughly speaking, purity is automatic for constrained absolutely continuous commuting row contractions if $V(\J)\cap \bS_d$ is small enough. 

Recall that a closed subset $K\subset \bS_d$ is said to be \emph{$\A_d$--totally null} if $|\eta|(K)=0$ for every regular Borel measure $\eta$ on $\bS_d$ which induces an $\A_d$--Henkin functional via integration.

\begin{theorem}\label{T:constrainedpure}
Let $\J\subset \A_d$ be a closed ideal with the property that $V(\J)\cap \bS_d$ is an $\A_d$--totally null set.
If $T = (T_1,\dots,T_d)$ is an absolutely continuous $\J$-constrained commuting row contraction on a separable Hilbert space,  then $T$ is pure.
\end{theorem}

\begin{proof}
By Corollary \ref{C:abscontchar}, we have that $T$ has a co-extension of the form 
$Z^{(\kappa)} \oplus W$ for some cardinal $\kappa$ and a spherical unitary
$W$ with an $\A_d$--Henkin spectral measure $\eta$ supported on $V(\J)\cap \bS_d$. Since this set is assumed to be $\A_d$--totally null, we see that  $\eta=0$ and conclude that $W=0$, and thus $T$ is pure.
\end{proof}

Recall that purity is equivalent to the fact that the minimal co-extension is given merely by a multiple of the $d$-shift $M_z$. It is interesting that the previous result encodes this dilation theoretic rigidity in the size of the common zero set of the ideal $\J$ on the sphere. 

We close the paper with some further remarks concerning Theorem \ref{T:constrainedpure} and the condition of $V(\J)\cap\bS_d$ being totally null. 

First, let us exhibit a class of examples of ideals $\J\subset \A_d$ with that property. Let $K\subset \bS_d$ be a closed $\A_d$--totally null set. By \cite[Theorem 9.5]{CD}, there exists $\phi\in \A_d$ such that $\phi|_K=1$ and $|\phi(\zeta)|<1$ for $\zeta\in \bS_d\setminus K$. Then, we see that $\phi(z)=1$ for $z\in \ol{\bB_d}$ if and only if $z\in K$, so that if we put $\J=\ol{(1-\phi)\A_d}$, then we have $V(\J)=K$. Unfortunately, the closure $\J_w$ of such an ideal tends to be too big for the conclusion of Theorem \ref{T:constrainedpure} to be valuable  (see \cite[Corollary 3.5]{ideals}). 

In order to construct ideals for which Theorem \ref{T:constrainedpure} applies in an interesting way, one can proceed as follows. Let $K\subset \bS_d$ be a closed $\A_d$--totally null subset and let $\Lambda\subset \bB_d$ be an \emph{interpolating sequence} for $\M_d$ with the property that $\ol{\Lambda}\cap \bS_d\subset K$. Such sequences always exist, see \cite[Proposition 9.1]{DHS15}. By \cite[Corollary 5.11]{ideals}, there is a closed ideal $\J\subset \A_d$ such that $V(\J)=\Lambda\cup K$. It is easily verified then that $\J_w$ consists of functions in $\M_d$ vanishing on $\Lambda$ and thus is a proper ideal.

In the single operator case, we note that no condition other than non-triviality needs to be imposed on the ideal $\J\subset \A_1=A(\bD)$ annihilating $T$ to guarantee purity. Indeed, in this case  it is known that $V(\J)\cap \bT$ has Lebesgue measure $0$ (see \cite{Hoffman}, \cite{Rud57}), and thus that set must be $\A_1$-totally null (see \cite[Chapter 10]{Rudin}).

Finally, we provide an example that illustrates that in contrast with the single operator case, some condition must be imposed on the annihilating ideal $\J$ 
in order for Theorem \ref{T:constrainedpure} to hold for $d\geq 2$. Let $U$ be the unitary operator of multiplication by the variable on the space $L^2(\bS_1)$. 
Consider the commuting row contraction $T=(U,0)$ which is clearly absolutely continuous. 
It is obvious that $T$ is not pure since $U$ is not. On the other hand, note that a function $f\in \A_2$ annihilates $T$ 
if and only if $f \in \J=\ol{z_2\A_d}$. The issue here is that
$$
V(\J)\cap \bS_2=\{(\zeta,0):\zeta\in \bS_1\}
$$
is not $\A_2$--totally null: it is a big circle supporting the one dimensional Lebesgue measure which is easily seen to be an $\A_2$--Henkin measure. This unfortunate phenomenon may be a reflection of the complexity of the structure of the closed ideals in $\A_d$, a study of which is undertaken in the upcoming paper \cite{ideals}.

\bibliographystyle{amsplain}

\begin{thebibliography}{99}


\bibitem{Arv98} W. Arveson, 
\textit{Subalgebras of $C^*$-algebras III: Multivariable operator theory}, 
Acta Math.\ \textbf{181} (1998), 159--228.


\bibitem{athavale} A. Athavale, 
\textit{On the intertwining of joint isometries}, 
J. Operator Theory\ \textbf{23} (1990), 339--350.


\bibitem{bercovici} H. Bercovici, 
\textit{Operator theory and arithmetic in {$H^\infty$}}, Mathematical Surveys and Monographs  \textbf{26}, American Mathematical Society, Providence, RI (1988).


\bibitem{BCP} S. Brown, B. Chevreau and C. Pearcy, 
\textit{Contractions with rich spectrum have invariant subspaces}, 
J. Operator Theory  \textbf{1} (1979), 123--136


\bibitem{blecherlemerdy} D. Blecher and C. Le Merdy,
\textit{Operator algebras and their modules—an operator space approach}, 
London Mathematical Society Monographs. New Series \textbf{30}, 
Oxford Science Publications, Oxford (2004).


\bibitem{chen03} Z. Chen,  
\textit{Characterizations of Arveson's Hardy space},
Complex Var.\ Theory Appl. \textbf{48} (2003), 453--465. 


\bibitem{CD} R. Clou\^atre and K. Davidson,
\textit{Duality, convexity and peak interpolation in the Drury-Arveson space},
Adv. Math. 295, (2016), 90--149.


\bibitem{ideals} R. Clou\^atre and K. Davidson,
\textit{Ideals in a multiplier algebra on the ball},
preprint, 2015.


\bibitem{ColeRange} B. Cole and M. Range, 
\textit{{$A$}-measures on complex manifolds and some applications},
J. Funct. Anal. \textbf{11} (1972), 393--400. 
 
 
 \bibitem{CER} F. Cunningham Jr,  E. Effros and N. M Roy, 
\emph{$M$-structure in dual Banach spaces}, 
Israel J. Math.\ \textbf{14} (1973), 304--308. 
 
 
\bibitem{DAELW} G. H. Dales, Aiena, P., J. Eschmeier, K. Laursen and G. Willis,
\textit{Introduction to Banach algebras, operators, and harmonic analysis}, London Mathematical Society Student Texts,   \textbf{57}, Cambridge University Press, Cambridge (2003).

\bibitem{DHS15} K. Davidson, M. Hartz and O. Shalit,
\textit{Multipliers of Embedded Discs},
Complex Anal.\ Oper.\ Theory \textbf{9} (2015), 287--321.

\bibitem{DP98b} K. Davidson and D. Pitts, 
\textit{Nevan\-linna-Pick interpolation for non-commutative analytic Toeplitz algebras}, 
Integral Equations Operator Theory \textbf{31} (1998), 321--337. 





\bibitem{DLP} K. Davidson, J. Li and D. Pitts, 
\textit{Absolutely continuous representations and a Kaplansky density theorem for free semigroup algebras},
J. Funct. Anal. \textbf{224} (2005), no. 1, 160--191.




\bibitem{Drury78} S. Drury, 
\textit{A generalization of von Neumann's inequality to the complex ball}, 
Proc.\ Amer.\ Math.\ Soc. \textbf{68} (1978), 300--304.


\bibitem{eschmeier95} J. Eschmeier, 
\textit{$H^\infty$-functional calculus for spherical contractions}, preprint, University of Leeds, 1995.


\bibitem{eschmeier97} J. Eschmeier, 
\textit{Invariant subspaces for spherical contractions}, 
Proc. London Math. Soc. (3) \textbf{75} (1997), 157--176.


\bibitem{eschmeier2015} J. Eschmeier, 
\textit{Private communication}, October 2015.


\bibitem{Henkin} G. Henkin,
\textit{The Banach spaces of analytic functions in a ball and in a bicylinder are nonisomorphic},
Funkcional. Anal. i Prilo\v zen. \textbf{4} (1968), 82--91.


\bibitem{Hoffman} K. Hoffman,
\textit{Banach spaces of analytic functions}, Prentice Hall, 1962.




\bibitem{mccullough92} S. McCullough, 
\textit{Carath\'eodory interpolation kernels},
Integral Equations Operator Theory \textbf{15} (1992), 43--71. 


\bibitem{MV93} V. M\"uller and F. Vasilescu,
\textit{Standard models for some commuting multioperators},
Proc.\ Amer.\ Math.\ Soc.\ \textbf{117} (1993), 979--989. 


\bibitem{Pop06} G. Popescu,
\textit{Operator theory on noncommutative varieties},
Indiana Univ.\ Math.\ J. \textbf{55} (2006), 389--442. 


\bibitem{quiggen93} P. Quiggen,
\textit{For which reproducing kernel Hilbert spaces is Pick's theorem true?}
Integral Equations Operator Theory \textbf{16} (1993), 244--266.


\bibitem{Rud57} W. Rudin, 
\textit{The closed ideals in an algebra of analytic functions},
Canad.\ J. Math. \textbf{9} (1957), 426--434. 


\bibitem{Rudin} W. Rudin, 
\textit{Function Theory in the Unit Ball of $\bC^n$}, 
Springer-Verlag, 1980.


\bibitem{ShalitDAsurvey} O.M. Shalit,
\textit{Operator theory and function theory in Drury-Arveson space and its quotients},
Handbook of Operator Theory, D. Alpay. ed., pp. 1125--1180,
Springer Verlag, Basel 2015.


\bibitem{NagyFoias} B. Sz.-Nagy, C. Foias, H. Bercovici and L. Kerchy,
\textit{Harmonic analysis of operators on Hilbert space. Second edition}, 
Universitext  Springer, New York, 2010. xiv+474 pp.


\end{thebibliography}

\end{document}